\documentclass[11pt,a4paper]{amsart}
\usepackage{amsmath,amssymb, amsbsy}
\usepackage{color,psfrag}
\usepackage[dvips]{graphicx}

\newcommand{\cG}{{\mathcal G}}

\renewcommand{\phi}{\varphi}

\usepackage{enumerate}
\newcommand{\R}{{\mathbb R}}
\newcommand{\N}{{\mathbb N}}

\newcommand{\HH}{\mathcal{H}}

\renewcommand{\geq }{\geqslant}

\renewcommand{\leq }{\leqslant}
\def\neweq#1{\begin{equation}\label{#1}}
\def\endeq{\end{equation}}
\def\eq#1{(\ref{#1})}
\newtheorem{theorem}{Theorem}[section]
\newtheorem{proposition}[theorem]{Proposition}
\newtheorem{lemma}[theorem]{Lemma}
\newtheorem{corollary}[theorem]{Corollary}

\newtheorem{definition}[theorem]{Definition}
\theoremstyle{definition}

\begin{document}

\title[Symmetry in composite partially hinged plates]{About symmetry in partially hinged composite plates}

\author[Elvise BERCHIO]{Elvise BERCHIO}
\address{\hbox{\parbox{5.7in}{\medskip\noindent{Dipartimento di Scienze Matematiche, \\
Politecnico di Torino,\\ Corso Duca degli Abruzzi 24, 10129 Torino, Italy. \\[3pt]
\em{E-mail address: }{\tt elvise.berchio@polito.it}}}}}
\author[Alessio FALOCCHI]{Alessio FALOCCHI}
\address{\hbox{\parbox{5.7in}{\medskip\noindent{Dipartimento di Scienze Matematiche, \\
Politecnico di Torino,\\ Corso Duca degli Abruzzi 24, 10129 Torino, Italy. \\[3pt]
\em{E-mail address: }{\tt alessio.falocchi@polito.it}}}}}


\keywords{composite plate problem; partially hinged plate; Green function; polarization}

\subjclass[2010]{35J08, 35P05, 74K20}

\begin{abstract}
We consider a partially hinged composite plate problem and we investigate qualitative properties, e.g. symmetry and monotonicity, of the eigenfunction corresponding to the density minimizing the first eigenvalue. The analysis is performed by showing related properties of the Green function of the operator and by applying polarization with respect to a fixed plane. As a by-product of the study, we obtain a Hopf type boundary lemma for the operator having its own theoretical interest. The statements are complemented by numerical results.
\end{abstract}

\maketitle

\section{Introduction}
Let $\Omega=(0,\pi)\times(-\ell,\ell)\subset\R^2\, 
$ with $\ell>0$, we consider the weighted eigenvalue problem:
\begin{equation}\label{weight}
\begin{cases}
\Delta^2 u=\lambda\, p(x,y) u & \qquad \text{in } \Omega \\
u(0,y)=u_{xx}(0,y)=u(\pi,y)=u_{xx}(\pi,y)=0 & \qquad \text{for } y\in (-\ell,\ell)\\
u_{yy}(x,\pm\ell)+\sigma
u_{xx}(x,\pm\ell)=u_{yyy}(x,\pm\ell)+(2-\sigma)u_{xxy}(x,\pm\ell)=0
& \qquad \text{for } x\in (0,\pi)\, ,
\end{cases}
\end{equation}
where $\sigma\in[0,1)$ and, for $\alpha,\beta \in(0,+\infty)$ with $\alpha<\beta$ fixed, $p$ belongs to the following family of weights:
\begin{equation} \label{eq:famiglia}
P_{\alpha, \beta}:=\left\{p\in L^\infty(\Omega): \alpha\leq p\leq\beta \ \text{a.e. in } \Omega \
\text{ and }\int_{\Omega}p\,dxdy=|\Omega| \, \right\} \, .
\end{equation}
The interest for problem \eqref{weight} is due to the fact that it describes the oscillating modes of the non-homogeneous partially hinged rectangular plate $\Omega$ which, up to scaling, can model the decks of footbridges and suspension bridges, see \cite{bebuga2,fergaz,bookgaz} for more details; in particular, the partially hinged boundary conditions reflect the fact that decks of bridges are supported by the ground only at the short edges. We also remark that, in this framework, $\sigma$ represents the so-called Poisson ratio which for most materials belongs to the interval $[0,1)$, $p$ represents the density function of the plate and the integral condition in \eqref{eq:famiglia} means that the total mass of the plate is preserved. \par
In order to study the stability properties of the plate it is important to investigate the effect of the density function $p$ on the eigenvalues, i.e on the frequencies of the plate. In this respect, the starting point of the study is the minimization problem:
\begin{equation}\label{pbintro}
\inf_{p \in P_{\alpha, \beta}} \, \lambda_1(p)\,,
\end{equation}
where $\lambda_1(p)$ denotes the first eigenvalue of \eqref{weight}.
There exists a rich literature dealing with the second order Dirichlet version of \eqref{weight}-\eqref{pbintro} which is usually named \emph{composite membrane} problem; this corresponds to the problem of building a body of prescribed shape and mass, out of given materials in such a way that the first frequency of the resulting membrane is as small as possible, see e.g. \cite{chanillo1}-\cite{chanillo3} and the monograph \cite{henrot}.  In the fourth order case, problem \eqref{pbintro} is named \emph{composite plate} problem and has been mainly studied under clamped (Dirichlet) or hinged (Navier) boundary conditions, see e.g. \cite{anedda2},\cite{chen}-\cite{cuccu1}, \cite{lapr}. As far as we are aware, the \emph{partially hinged} composite plate problem \eqref{weight}-\eqref{pbintro} has only been studied in \cite{befafega}, see also \cite{befa} for results about higher eigenvalues; in \cite{befafega} it is proved that the infimum in \eqref{pbintro} is achieved by the piecewise constant density:
$$\widehat{p}(x,y) = \alpha \chi_{S} (x,y)+ \beta \chi_{\Omega \setminus S}(x,y)\,,
$$
where $\chi_{S}$ denotes the characteristic function of a suitable set $S \subset \Omega$, see Proposition \ref{thm-exist-qual} in Section \ref{main}. This information is useful in engineering applications, since the assemblage of two materials with constant density is simpler than the manufacturing of a material having variable density; however, the region $S$ is given in terms of sub and super level sets of the eigenfunction $u_{\widehat p}$ of $\lambda_1(\widehat p)$ which is not explicitly known. Hence, in order to find more precise information about the location of the two materials,  it is important to study the qualitative properties of $u_{\widehat p}$ . 
In this field of research, typical results are qualitative properties, such as symmetry or monotonicity, of the first eigenfunction corresponding to the minimizer of \eqref{pbintro}, see e.g. \cite{anedda2}, \cite{chanillo} and references therein.
From this point of view, a crucial obstruction, when passing from the membrane to the plate problem, i.e. from the second to the fourth order case,
is represented by the loss of maximum and comparison principles which usually enter in the techniques applied to prove symmetry results, such as reflections methods or moving planes techniques. Nevertheless, some interesting results have been recently obtained in \cite{CV} and \cite{CV2} for the fourth order equation by exploiting suitable choices of the boundary conditions and of the geometry of the domain for which proper comparison principles hold (e.g. by considering Navier boundary conditions on sufficiently smooth domains or Dirichlet boundary conditions on balls). 
\par

Regarding problem \eqref{weight}, the above mentioned difficulties are further increased by the unusual boundary conditions and only few results about qualitative properties of the first eigenfunction were proved in \cite{befafega} for $p=p(y)$. An important step forward in the study of \eqref{weight} has been recently done in \cite{ppp} by computing explicitly the Fourier expansion of the Green function of the operator in \eqref{weight} and by showing its positivity, see Proposition \ref{monotonia} below. In particular, as a direct consequence of these results, it follows the positivity of the first eigenfunction of \eqref{weight} and the simplicity of the first eigenvalue which are not obvious facts when dealing with higher order PDEs. The main aim of the present paper is to investigate reflection and monotonicity properties of the Green function in order to, possibly, exploit them to deduce related properties of the eigenfunction $u_{\widehat p}$. Broadly speaking, the idea is to replace maximum principle arguments, not available in this case, with arguments based on the explicit knowledge of the Green function. To our best knowledge, this idea was first exploited in \cite{begawe} and \cite{fgw}; in particular, in \cite{begawe} a variant of the moving plane method, relying on fine estimates for the Green function \cite{bo}, was developed in order to prove Gidas-Ni-Nirenberg type symmetry results for higher order Dirichlet problems in the ball.  A similar approach has also been recently adopted in \cite{CV} for the Dirichlet composite plate problem in the ball, in order to prove radial symmetry and monotonicity of the first eigenfunction. We notice that all above mentioned proofs are based on \emph{polarization}, a simple two-point
rearrangement for functions which is well defined in first order Sobolev spaces, spaces of continuous functions or $L^p$-spaces, see e.g.
\cite{bartsch-weth-willem,brock,brock-solynin,smets-willem,weth}. 

However, since the Green function to \eqref{weight} is only known in terms of its Fourier expansion, it is hard to get in our case all the precise information available for the Green function of the Dirichlet problem in balls, see \cite[Chapter 6]{book}, and, in turn, to adopt in our framework the moving plane method as done in \cite{CV}. Nevertheless, we still managed to apply polarization by fixing the plane of reflection equal to the line $x=\pi/2$ and by exploiting suitable reflection properties proved for the Green function with respect to this line. More precisely, we first establish a duality
principle which reduces our minimization
problem in $H^2$ to a maximization problem in
$L^{2}$ and then, with the help of polarization, we prove a partial symmetry result in the $x$-direction for the maximizers of the reduced problem. We remark that, in general, it is quite delicate to exploit polarization in the higher order case since the
polarization of an $H^2$-function is not contained in $H^2$
anymore; the duality principle helps us to overcome this difficulty, see Lemma \ref{dualityprinciple}. We refer the interested reader to \cite{begawe} where a similar idea was originally exploited to prove partial symmetry of minimizers for subcritical higher order Sobolev embeddings into weighted $L^p$ spaces. Unfortunately, the fact of reflecting with respect to a fixed plane does not allow us to get monotonicity information about $u_{\widehat p}$ as it happens, instead, when applying the moving plane method. However, by a direct inspection of the Green function derivatives we succeed in deducing some local information about the derivatives of $u_{\widehat p}$. It is worth mentioning that, as a by-product of our analysis, we also obtain a Hopf type boundary lemma for the operator \eqref{weight} having it own theoretical interest, see Corollary \ref{hopfcor}.\par \smallskip\par

The paper is organised as follows: in Section \ref{main} we set precisely our problem an we state our main results while in Section \ref{num1} we complement the study with suitable numerical results. The other sections are devoted to the proofs of the results.
\section{Main results}\label{main}

The natural functional space where to set problem \eq{weight} is
$$
H^2_*(\Omega)=\big\{u\in H^2(\Omega): u=0\mathrm{\ on\ }\{0,\pi\}\times(-\ell,\ell)\big\}\,.
$$
Note that the condition $u=0$ has to be meant in a classical sense because $\Omega$ is a planar domain and the energy space $H^2_*(\Omega)$ embeds into continuous functions. 
Furthermore, for $\sigma\in[0,1)$ fixed, $H^2_*(\Omega)$ is a Hilbert space when endowed with the scalar product
$$
(u,v)_{H^2_*(\Omega)}:=\int_\Omega \left[\Delta u\Delta v+(1-\sigma)(2u_{xy}v_{xy}-u_{xx}v_{yy}-u_{yy}v_{xx})\right]\, dx \, dy \,
$$
with associated norm
$
\|u\|_{H^2_*(\Omega)}^2=(u,u)_{H^2_*(\Omega)} \, 
$
which is equivalent to the usual norm in $H^2(\Omega)$, see \cite[Lemma 4.1]{fergaz}. Problem \eq{weight} in weak form reads
\begin{equation}
\label{eigenweak1}
(u,\varphi)_{H^2_*(\Omega)} =\lambda(p\, u,\varphi )_{L^2(\Omega)} \qquad\forall \varphi\in H^2_*(\Omega).
\end{equation}
Hence, the first eigenvalue can be characterized as follows:
\begin{equation}\label{lambdaP}
\lambda_1(p) :=  \, \min_{u \in H^2_*(\Omega) \setminus\{0\}} \frac{\|u\|_{H^2_*}^2}{\|\sqrt{p}\,u\|_{2}^2}.
\end{equation}
It is well known that the sign and simplicity property of the first eigenfunction of a differential operator are strictly related to the sign property of its Green function. For $p=(\rho,w)\in\overline \Omega$ fixed, the \emph{Green function} to the operator in \eqref{weight} is, by definition, the unique solution $G(\cdot,p)\in H^2_*(\Omega)$ to:
	\begin{equation*}
	(G(\cdot,p),\varphi)_{H^2_*(\Omega)} =\langle \delta_{p},\varphi\rangle=\varphi(p) \qquad\forall \varphi\in H^2_*(\Omega)\,
	\end{equation*}
	and it has been recently computed in \cite{ppp}; we recall the precise statement here below.
\begin{proposition}\label{monotonia} \cite{ppp}
	There holds
	\begin{equation}\label{green}
	G(x,y,\rho,w)= \dfrac{1}{2\pi}\sum_{m=1}^{+\infty}\dfrac{\phi_m(y,w)}{m^3}\sin(m\rho)\,\sin(mx) \qquad \forall (x,y)\in\overline \Omega\quad\forall (\rho,w)\in\overline \Omega,
	\end{equation}
	where the $\phi_m\in C^2([-\ell, \ell]\times [-\ell, \ell])$ are strictly positive and strictly decreasing with respect to $m$, i.e.
	\begin{equation}\label{decreasing}
	0<\phi_{m+1}(y,w)<\phi_m(y,w)\qquad\forall m\in\mathbb{N^+},\forall y,w\in[-\ell,\ell]\,.
	\end{equation}
	Furthermore, $G\in C^0(\overline \Omega\times \overline \Omega)$ and 
	\begin{equation}\label{tesi}
	G(x,y,\rho,w)> 0\qquad \forall (x,y) \in (0,\pi)\times[-\ell,\ell]\quad  \forall (\rho,w) \in (0,\pi)\times[-\ell,\ell].
	\end{equation}
\end{proposition}
For the explicit (and very involved) expression of the functions  $\phi_m$, we refer the interested reader to \cite{ppp}. In the present paper we enrich the statement of Proposition \ref{monotonia} by showing that:
	\begin{theorem}\label{greenmonotonia}
	 For all $y,w \in [-\ell,\ell]$, there holds:
	\begin{equation*}\label{tesi1}
	\begin{split}
	&\, G_x(0,y,\rho,w)> 0\quad \text{ and } \quad G_x(\pi,y,\rho,w)< 0\qquad \quad\forall\rho\in(0,\pi) ;\\
	&\, G_x\bigg(\frac{\pi}{2},y,\rho,w\bigg)<0\quad \forall \rho\in\bigg(0,\frac{\pi}{2}\bigg)\,;\quad G_x\bigg(\frac{\pi}{2},y,\frac{\pi}{2},w\bigg)=0\,;\quad G_x\bigg(\frac{\pi}{2},y,\rho,w\bigg)>0\quad \forall \rho\in\bigg(\frac{\pi}{2}, \pi\bigg)\,,
		\end{split}
	\end{equation*}
	where the derivative of $G$ in the $x$-direction are meant in classical sense and $G_x\in  C^0(\overline \Omega\times \overline \Omega)$.
	Furthermore, since $G(x,y,\rho,w)=G(\rho,y,x,w)$, the above results hold by inverting $x$ and $\rho$.

\end{theorem}
It's worth pointing out that neither the proof of \eqref{tesi} or that of Theorem \ref{greenmonotonia} trivially follow from \eqref{green}; indeed, they require an accurate inspection of each term of the expansion and sharp estimates. In this respect, the hardest part is the proof of \eqref{decreasing} which follows only after lengthy computations.

A remarkable consequence of Proposition \ref{monotonia} is the validity of the positivity preserving property for the operator in \eqref{weight} whereas Theorem \ref{greenmonotonia} can be exploited to prove a Hopf type boundary lemma. For the sake of clarity we collect both statements in the following:
\begin{corollary}\label{hopfcor}
If $f\in L^2(\Omega)$ and $u\in H^2_*(\Omega)$ is a (weak) solution to
\begin{equation*}
\begin{cases}
\Delta^2 u=f & \qquad \text{in } \Omega \\
u(0,y)=u_{xx}(0,y)=u(\pi,y)=u_{xx}(\pi,y)=0 & \qquad \text{for } y\in (-\ell,\ell)\\
u_{yy}(x,\pm\ell)+\sigma
u_{xx}(x,\pm\ell)=u_{yyy}(x,\pm\ell)+(2-\sigma)u_{xxy}(x,\pm\ell)=0
& \qquad \text{for } x\in (0,\pi)\, ,
\end{cases}
\end{equation*}

then the following implication holds
\begin{equation}\label{tesi0}
	f\geq 0,\,\, f\not\equiv 0  \text{ in } \Omega\quad  \Rightarrow  \begin{cases} & u>0  \text{ in } (0,\pi)\times[-\ell,\ell];\\
	&u_{x}(0, y)>0 \text{ and } u_{x}(\pi, y)<0\,\,\, \forall\, y\in [-\ell, \ell]\,, \end{cases}
	\end{equation}
	where the derivatives of $u$ in \eqref{tesi0} are meant in classical sense.
\end{corollary}

Coming back to problem  \eq{weight}, in what follows we will always assume $$0<\alpha<1<\beta\,.$$ By exploiting Proposition \ref{monotonia} and Theorem \ref{greenmonotonia}, we also obtain the following statement about the first eigenfunction of \eqref{weight}:
\begin{corollary}\label{pppeig}
Let $p\in P_{\alpha,\beta}$ with $P_{\alpha,\beta}$ as in \eqref{eq:famiglia}. Then, the first eigenvalue $\lambda_1(p)$ of problem \eqref{weight}
 is simple and the first eigenfunction $u_{p}$ is of one sign in $\Omega$. Furthermore, $u_p \in C^{3,\gamma}(\overline \Omega)$ for some $0<\gamma <1$ and, assuming $u_p$ positive, we have: \begin{equation}\label{hopf}
 (u_p)_{x}(0, y)>0 \quad \text{and} \quad (u_p)_{x}(\pi, y)<0 \quad  \forall\, y\in [-\ell, \ell]\,.
 \end{equation}
\end{corollary}
Next we set
\begin{equation}\label{CP}
\lambda_{\alpha,\beta} :=\inf_{p \in P_{\alpha,\beta}} \, \lambda_1(p)\,.
\end{equation}
\begin{definition}\label{def-opt}
	A couple $ (\widehat{p},\widehat u) \in P_{\alpha,\beta} \times H^2_*(\Omega)$ is called {\em optimal pair} if $\widehat p$ achieves the infimum in \eqref{CP} and $\widehat u$ is an eigenfunction associated with $\lambda_1(\widehat p)$ .
\end{definition}

From \cite[Theorem 3.2]{befafega}, suitably combined with Corollary \ref{pppeig}, we have the following:
\begin{proposition}\label{thm-exist-qual}
 \cite{befafega}	There exists and optimal pair $(\widehat{p},\widehat{u}) \in P_{\alpha,\beta} \times H^2_*(\Omega)$ with $\widehat{u}$ positive. Furthermore, 
 	\begin{equation}
	\label{pS}
	\widehat{p}(x,y)=p_{\widehat{u}}(x,y):= \alpha \chi_{S} (x,y)+ \beta \chi_{\Omega \setminus S}(x,y)\,\quad \text{for a.e. }\quad (x,y)\in \Omega\,,
	\end{equation}
	where $\chi_{S}$ and $ \chi_{\Omega \setminus S}$ are the characteristic functions of the sets $S$ and $\Omega \setminus S$; $S\subset \Omega$
	is such that $|S|=\frac{\beta-1}{\beta-\alpha}\,|\Omega|$ and $S := \{ (x,y)\in \Omega\,:\,0<\widehat{u}(x,y) \leq \sqrt{t} \}$ for some $t> 0$.
\end{proposition}
%
%

Proposition \ref{thm-exist-qual} gives the useful information that optimal plates, in the sense of Definition \ref{def-opt}, are made by only two materials. However, the region $S$ is given in terms of the optimal eigenfunction $\widehat{u}$ which is not explicitly known, hence, in order to locate the position of the materials, it is important to investigate symmetry and monotonicity properties of $\widehat{u}$. To this aim, we set 
 \begin{equation}\label{iperp}
\mathcal{H}_{\frac{\pi}{2}}:=\bigg\{(x,y)\in\mathbb{R}^2: x\leq \frac{\pi}{2}\bigg\}
 \end{equation}
and we denote by $(\overline x, y)\in \mathbb{R}^2$ the reflection of $(x,y)\in\mathbb{R}^2$ with respect to $\partial\mathcal{H}_{\frac{\pi}{2}}$, i.e. $\overline x=\pi-x$. By exploiting related reflection properties of the Green function, see Lemma \ref{lemma0} in Section \ref{6}, we prove:
\begin{theorem}\label{partialsym}
	Let $(\widehat{p},\widehat{u}) \in P_{\alpha,\beta} \times H^2_*(\Omega)$ be an optimal pair with $\widehat{u}$ positive. Then, one of the following alternative holds:
 \begin{itemize}
	\item[$(i)$] $\widehat{u}(x,y)>\widehat{u}(\overline x,y)$ for all $(x,y) \in (0, \frac{\pi}{2})\times[-\ell,\ell]\,;$
	\item[$(ii)$] $\widehat{u}(x,y)<\widehat{u}(\overline x,y)$ for all $(x,y) \in (0, \frac{\pi}{2})\times[-\ell,\ell]\,;$
	\item[$(iii)$] $\widehat{u} (x,y)= \widehat{u}(\overline x,y)$  for all $(x,y) \in [0,\pi]\times[-\ell,\ell]$. 
	\end{itemize}	
	
\end{theorem}

In few words, according to Theorem \ref{partialsym}, two situations may occur: either $u_{\widehat p}$ is symmetric w.r.t. $\partial\mathcal{H}_{\frac{\pi}{2}}$, i.e. $(iii)$ occurs, or it is ``concentrated" on one of the two half plates delimited by $\partial\mathcal{H}_{\frac{\pi}{2}}$, i.e. $(i)$ or $(ii)$ occurs. In particular, if $(i)$ occurs then, by symmetry, we can always find an optimal pair such that also $(ii)$ occurs and uniqueness of the optimal pair certainly fails. We notice that uniqueness is even not guaranteed in case $(iii)$ since there could exist many weights symmetric with respect to $x=\frac{\pi}{2}$ and having the form \eqref{pS}; this case could be ruled out by showing very precise monotonicity information about $\widehat{u}$. Unfortunately, the polarization approach adopted in the proof of Theorem \ref{partialsym}, by keeping the reflection plane fixed, nothing says about the monotonicity of $\widehat{u}$; nevertheless, by direct inspection of the representation formula of solutions, we get the following local information.

\begin{proposition}
	\label{lemma2}
Let $(\widehat{p},\widehat{u}) \in P_{\alpha,\beta} \times H^2_*(\Omega)$ be an optimal pair with $\widehat{u}$ positive.Then, $\widehat{u}$ satisfies \eqref{hopf} and one among the following:
	\begin{itemize}
		\item[-] if case $(i)$ of Theorem \ref{partialsym} holds, then $\widehat{u}_x\big(\frac{\pi}{2},y\big)<0$ for all $y\in[-\ell,\ell]$;
		\item[-] if case $(ii)$ of Theorem \ref{partialsym} holds, then $\widehat{u}_x\big(\frac{\pi}{2},y\big)>0$ for all $y\in[-\ell,\ell]$;
		\item[-] if case $(iii)$ of Theorem \ref{partialsym} holds, then $\widehat{u}_x\big(\frac{\pi}{2},y\big)=0$ for all $y\in[-\ell,\ell]$.
	\end{itemize}	
\end{proposition}
For what so far stated, piecewise constant densities symmetric with respect to $x=\frac{\pi}{2}$ and with the denser material $\beta$ located near this line are among the candidates for being optimal in the sense of Definition \ref{def-opt}. Nevertheless, due to the high complexity of the analytic expression of the coefficients in \eqref{green}, a theoretical proof of their optimality seems out of reach by means of our techniques; this issue is instead supported by the numerical results we provide in Section \ref{num1}.

We conclude the section by pointing out that, even in the second order case, the picture of results about symmetry and monotonicity properties of minimizers of Poincar\'e inequalities on rectangular domains is far from being complete, when mixed boundary conditions are dealt with. See e.g. \cite[Section 6]{arioli} where the authors left as on open problem the one dimensionality of extremals for certain Poincar\'e inequalities arising when dealing with the stationary Navier-Stokes equation in a square, under mixed Dirichlet-Neumann boundary conditions. For results in this direction, but under Neumann boundary conditions, we refer the interested reader to \cite[Chapter II.5]{kaw}, \cite{naz} and references therein.

\section{Numerical results}\label{num1} 
In this section we illustrate some numerical results which complete the statements of Theorem \ref{partialsym} and Proposition \ref{lemma2}.

\subsection{Numerical algorithm to solve \eq{CP}} In order to find an optimal weight, we adopt an algorithm based on the following rearrangement lemma.

\begin{lemma}\label{rearrangement} \cite[Lemma 5.4]{befa}
	Let $u\in H^2_*(\Omega)$ be strictly positive in $\Omega$. Then, the problem
	$$M_{\alpha, \beta}:=\sup_{p\in P_{\alpha,\beta} } \int_{\Omega} p(x,y)u^2\,dx\,dy$$
	admits the solution $
	p_u(x,y) = \alpha \chi_{S} (x,y)+ \beta \chi_{\Omega \setminus S}(x,y)$ for a.e. $(x,y)\in \Omega$,
	where $S=S(u) \subset \Omega$
	is such that $|S|=\frac{\beta-1}{\beta-\alpha}\,|\Omega|$. Moreover, set
	$$
	t:= \sup \left\{ s > 0 : |\{(x,y)\in \Omega\,: 0<u(x,y) \leq \sqrt{s} \}| < \frac{\beta-1}{\beta-\alpha}\,|\Omega|\right\},
	$$
	we have that
	$
	 \{(x,y)\in \Omega\,:0<u(x,y) <  \sqrt{t} \} \subseteq S \subseteq \{(x,y)\in \Omega\,: 0<u(x,y)\leq  \sqrt{t}\}\,.
	$

\end{lemma}
To solve \eqref{CP} we run the numerical scheme below adjusted from \cite{befa}, see also \cite{chen} where the algorithm was proposed for the clamped and simply supported problems and \cite{chanillo} for related numerical results in the second order case.
\begin{itemize}
	\item[(i)] We solve numerically \eqref{weight} with an arbitrary weight $p^{(i)}$ and we determine the corresponding first eigenvalue $\lambda^{(i)}_1$ and the first eigenfunction $u^{(i)}$. 
	\item[(ii)] We compute numerically $t^{(i)}>0$ such that $|S^{(i)}|=| \{(x,y)\in \Omega\,:0<u^{(i)}(x,y)\leq\sqrt{ t^{(i)}}\}|=\frac{\beta-1}{\beta-\alpha}|\Omega|$ and we define the weight $$p^{(i+1)}:=p_{u^{(i)}}=\alpha \chi_{S^{(i)}} (x,y)+ \beta \chi_{\Omega \setminus S^{(i)}}(x,y).$$
	\item[(iii)] We solve numerically \eqref{weight} with the weight $p^{(i+1)}$ and we determine the corresponding first eigenvalue $\lambda^{(i+1)}_1$ and the first eigenfunction $u^{(i+1)}$. 
	\item[(iv)] Thanks to Lemma \ref{rearrangement} we get 
	\begin{equation*}\label{algoritmo}
	\|\sqrt{p^{(i+1)}}u^{(i)}\|_2^2\geq \|\sqrt{p^{(i)}}u^{(i)}\|_2^2.
	\end{equation*}
	Notice that we can apply Lemma \ref{rearrangement} with $S=S^{(i)}$ as in step (ii) since the $u^{(i)}$ solve the equation in \eqref{weight} a.e. hence, being strictly positive, their level sets must have zero measure.
	\item[(v)] We use the characterization \eqref{lambdaP}
	\small	\begin{equation*}
	\lambda^{(i+1)}_1= \min_{\substack{u \in H^2_*(\Omega)\setminus\{0\}}}
	\frac{\|u\|_{H^2_*}^2}{\|\sqrt{p^{(i+1)}}\,u\|_{2}^2}=\frac{\|u^{(i+1)}\|_{H^2_*}^2}{\|\sqrt{p^{(i+1)}}\,u^{(i+1)}\|_{2}^2}\leq \frac{\|u^{(i)}\|_{H^2_*}^2}{\|\sqrt{p^{(i+1)}}\,u^{(i)}\|_{2}^2}\leq \frac{\|u^{(i)}\|_{H^2_*}^2}{\|\sqrt{p^{(i)}}\,u^{(i)}\|_{2}^2}=\lambda^{(i)}_1.
	\end{equation*}
	\normalsize
	\item[(vi)]  Iterating the procedure, we obtain a non increasing sequence $i\mapsto\lambda_1^{(i)}$ bounded from below by $\lambda_{\alpha,\beta}$, so that the convergence of the algorithm to a certain $\overline\lambda_1\geq \lambda_{\alpha,\beta}$ is assured.
\end{itemize}
The only drawback of this algorithm is that we do not know a priori whether $\overline \lambda_1= \lambda_{\alpha,\beta}$; from a numerical point of view the problem may be circumvented by repeating the procedure with several different initial weights and noticing that we always get the same limit. 
To find the approximate solution of \eqref{weight}, for a given weight, we expand the solutions in Fourier series, adopting as orthonormal basis of $L^2(\Omega)$ the explicit eigenfunctions of \eqref{weight} with $p\equiv1$, known from \cite{fergaz}; in order to get a numerical approximation, we truncate the series at a certain $N \in \N_{+}$ and we solve a linear system of $2N$ equations where the unknowns are the Fourier coefficients, see \cite{befa} for the details.

\subsection{Conclusions }

\begin{figure}[!hbt]
	\centering
	{\includegraphics[width=13cm]{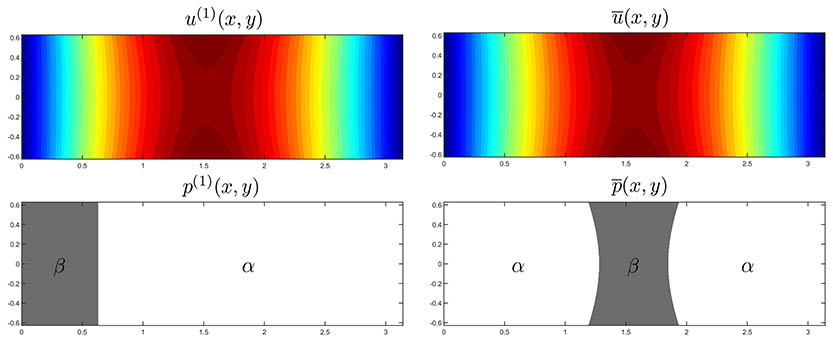}}
	\caption{The level sets of $u^{(1)}$ (left) and $\overline u$ (right) corresponding respectively to the densities $p^{(1)}(x,y)$ and $\overline p(x,y)$. We assume $N=20$ and \eqref{parameter}.}
	\label{fig}
\end{figure}

	\begin{figure}[!hbt]
	\centering
	{\includegraphics[width=10cm]{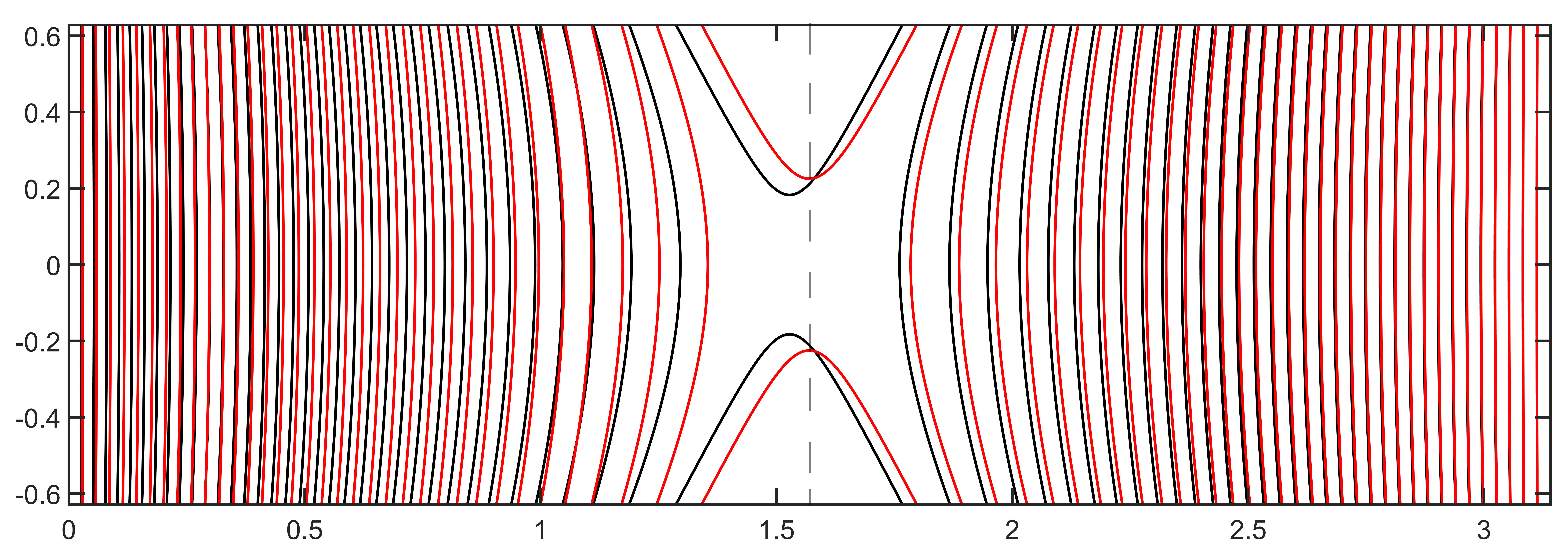}}
	\caption{A comparison betwen the level sets of $u^{(1)}$ (black) and $\overline u$ (red) corresponding respectively to the densities $p^{(1)}$ and $\overline p$}.
	\label{fig2}
\end{figure}

 In Figures \ref{fig} and \ref{fig2} we show some of our results on a plate having the following features:
\begin{equation}\label{parameter}
\sigma=0.2\,,\qquad \ell=\dfrac{\pi}{5}\,,\qquad\alpha=0.5\,,\qquad \beta=6\alpha.
\end{equation}
More precisely, in Figure \ref{fig} we compare the level sets of $u^{(1)}$ (left) and of $\overline u$ (right) corresponding, respectively, to the  weights $p^{(1)}$ and $\overline p$ plotted below, i.e. the initial and the last weight of our algorithm: even if $p^{(1)}$ and $u^{(1)}$ are not $\pi/2$-symmetric, $\overline p$ and $\overline u$ are $\pi/2$-symmetric. A comparison between the level sets of $u^{(1)}$ (black) and $\overline u$ (red) is given in Figure \ref{fig2}. We triggered the algorithm with very different initial weights $p^{(1)}$ either symmetric or not; after some iterations, we always get that the procedure converges to the same density $\overline p$ symmetric with respect to $x=\pi/2$. Furthermore, we repeated the experiments reducing the width of the plate $\ell$ towards choices consistent with common bridge design, e.g. $\ell=\frac{\pi}{150}$. In any case, we recorded the same kind of results which suggest to locate the denser material in the middle of the plate; we point out that by reducing $\ell$ we find weights $\overline p$ equals to $\beta$ on a region which is approximately a rectangle centered at $\pi/2$. \par
\smallskip\par

In conclusion, the observed results lead us to state the following conjecture:
\begin{center}
	\textbf{there exists an optimal pair $(\widehat p, \widehat u)$ of \eqref{CP} with $\widehat u$ symmetric w.r.t. $x=\pi/2$ and s.t.}
\end{center}
\begin{equation*}
\begin{split}
	&\widehat{u}_x (x,y)>0\quad \forall (x,y)\in \bigg(0,\dfrac{\pi}{2}\bigg)\times[-\ell,\ell]\,;\quad \widehat{u}_x (x,y)<0\quad \forall (x,y)\in \bigg(\dfrac{\pi}{2},\pi\bigg)\times[-\ell,\ell]\,;\\
	&\widehat{u}_y (x,y)>0\quad  \forall (x,y)\in (0,\pi)\times(0,\ell]\,;\quad \widehat{u}_y (x,y)<0\quad \forall (x,y)\in (0,\pi)\times[-\ell,0)\,.
\end{split}
\end{equation*}
Clearly, if the above conjecture holds, by taking $\widehat{p}=p_{\widehat{u}}$ as given in \eqref{pS}, we get that the corresponding optimal weight is symmetric w.r.t. $x=\pi/2$ and it is equal to $\beta$ in the central part of the plate.

\section{Proof of Theorem \ref{greenmonotonia}}\label{proof}
We define the series
\begin{equation}\label{seriedef}
\mathcal{S}_1(z):=\sum_{m=1}^{+\infty} \dfrac{\phi_m}{m^2}\sin(mz)\qquad\text{and}\qquad \mathcal{S}_2(z):=\sum_{m=1}^{+\infty} (-1)^m\dfrac{\phi_m}{m^2}\sin(mz)\,,
\end{equation}
with the $\phi_m$ as in \eqref{green}, and we state two preliminary lemmas. Notice that we neglect the dependence on $y,w$ since it does not play a role in the proofs. 
\begin{lemma}\label{lemma4}
	Let the series  $\mathcal{S}_1(z)$  be as in \eqref{seriedef}; then
	\begin{equation}\label{S1}
	\mathcal{S}_1(z)>0\qquad\forall z \in (0,\pi)\,.
	\end{equation}
\end{lemma}
\begin{proof}
We split the proof into three steps.\par \smallskip \par

\textbf{Step 1.}	
 Thanks to Theorem \ref{monotonia}-\eqref{decreasing} we know that $
0<\phi_{m}<\phi_1$ $\forall m>1$,
then we obtain

\begin{equation*}\label{dis1}
\bigg|\sum_{m=2}^\infty\dfrac{\phi_m}{m^2}\sin(mz)\bigg|\leq \phi_1\sum_{m=2}^\infty\dfrac{|\sin(mz)|}{m^2}\leq \phi_1 \sum_{m=2}^\infty\dfrac{1}{m^2}=\phi_1\bigg[\frac{\pi^2}{6}-1\bigg]\qquad \forall z\in(0,\pi)
\end{equation*}
and, in turn, that
\begin{equation}\label{ppp centrale2}
\mathcal{S}_1(z)\geq \phi_1\left[\sin(z)-\bigg(\frac{\pi^2}{6}-1\bigg)\right]\qquad \forall z\in(0,\pi).
\end{equation}
Since $\arcsin\big(\frac{\pi^2}{6}-1\big)<\frac{\pi}{4}$,
 through \eqref{ppp centrale2} we have 
\begin{equation}\label{eq11}
\mathcal{S}_1(z)>0\qquad\forall z\in\bigg[\dfrac{\pi}{4},\dfrac{3}{4}\pi\bigg].
\end{equation}

\textbf{Step 2.}
 We fix $N\geq 3$ and write $\mathcal{S}_1(z)=
\sum_{m=1}^N  \dfrac{\phi_m}{m^2}\sin(mz)+\sum_{m=N+1}^\infty  \dfrac{\phi_m}{m^2}\sin(mz)\,.
$
Then, we exploit the elementary inequality 
$
\sin(mz)>\sin(z)$, $\forall z\in\big(0,\frac{\pi}{N+1}\big)$ and $\forall m=2,\dots,N$ (see \cite[Lemma 6.3]{ppp} for a proof) and Theorem \ref{monotonia}-\eqref{decreasing} to get
\begin{equation}\label{dise2}
\sum_{m=1}^N  \dfrac{\phi_m}{m^2}\sin(mz)> \phi_N\sin(z) \sum_{m=1}^N  \dfrac{1}{m^2}\qquad \forall z\in\bigg(0,\dfrac{\pi}{N+1}\bigg).
\end{equation}
On the other hand, through Theorem \ref{monotonia}-\eqref{decreasing}, we get 
\begin{equation}\label{dise222}
\bigg|\sum_{m=N+1}^\infty  \dfrac{\phi_m}{m^2}\sin(mz)\bigg|\leq \sum_{m=N+1}^\infty  \dfrac{\phi_m}{m^2}|\sin(mz)|\leq \phi_{N}\sum_{m=N+1}^\infty  \dfrac{1}{m^2}.
\end{equation}
By combining \eqref{dise2} and \eqref{dise222} we infer 
\begin{equation}\label{ineq22}
\begin{split}
\mathcal{S}_1(z)&\geq\phi_N\left(\sum_{m=1}^N  \dfrac{1}{m^2}\right)\, \left[\sin z-C_N \right]\qquad \forall z\in \bigg(0,\dfrac{\pi}{N+1}\bigg)\,,
\end{split}
\end{equation}
where $$
C_N:=\left(\sum\limits_{m=N+1}^\infty  \frac{1}{m^2}\right) \left(\sum\limits_{m=1}^N  \frac{1}{m^2}\right)^{-1}\,.$$ Next we denote by $z_N$ the unique solution to the equation:
\begin{equation*}\label{sin_ineq2}
\sin( z)=\overline C_N\qquad z\in(0,\pi/2)\,;
\end{equation*}
the above definition makes sense for all $N\geq 1$ since the map $N\mapsto C_N$ is positive, strictly decreasing and $0<C_N<1$. We prove that
\begin{equation}\label{ts0}
z_N< \dfrac{\pi}{N+2}\qquad \forall N\geq 3\,.
\end{equation}
When $N=3$, $z_3\approx0.21<\frac{\pi}{5}$ and \eqref{ts0} follows. We complete the proof of \eqref{ts0} by showing that
\begin{equation}\label{ts1}
C_N< \sin\bigg(\dfrac{\pi}{N+2}\bigg)\qquad \forall N\geq 4\,.
\end{equation}
To this purpose we notice that $\sum\limits_{m=1}^N  \frac{1}{m^2}> 1$ and $\sum\limits_{m=N+1}^\infty  \frac{1}{m^2}< \int_N^\infty\frac{1}{x^2}\,dx=\frac{1}{N}$, implying that
\begin{equation}\label{CN}
C_N< \dfrac{1}{N}\qquad \forall N\geq 2.
\end{equation}
To tackle \eq{ts1} we use the estimate:
\begin{equation}\label{stimaseno}	
\sin(x)\geq\dfrac{3}{\pi}x\qquad \forall x \in\bigg(0,\dfrac{\pi}{6}\bigg].
\end{equation}
Combining this with \eq{CN}, \eq{ts1} follows by noticing that
$
\dfrac{1}{N}<\dfrac{3}{N+2}$ for all $N\geq 4$. Finally, in view of \eq{ts0}, we get 
\begin{equation}\label{eq12}
\mathcal{S}_1(z)>0\qquad\forall z\in\bigg[\dfrac{\pi}{N+2},\dfrac{\pi}{N+1}\bigg)\qquad \forall N\geq 3.
\end{equation}
Hence, by combining \eqref{eq11} with \eqref{eq12} written for all $3\leq N \leq \overline N$, we obtain 
$
\mathcal{S}_1(z)>0$ $\forall z\in\big[\frac{\pi}{\overline N+2},\frac{3}{4}\pi\big]$ and passing to the limit as  $\overline N\rightarrow+\infty$ we conclude that
\begin{equation}\label{eq19}
\mathcal{S}_1(z)>0\qquad\forall z\in\bigg(0,\dfrac{3}{4}\pi\bigg].
\end{equation}

\textbf{Step 3.}
It remains to consider $z\in\big(\frac{3}{4}\pi,\pi\big)$.
For $N\geq3$, odd integer, we set $\overline z=\pi-z$ and we rewrite the series as
\begin{equation}\label{S1overz}
\begin{split}
\mathcal{S}_1(\overline z)=\sum_{\substack{
		m=1\\
		\text{odd}}}^N\bigg[ \dfrac{\phi_m}{m^2}\sin(m\overline z)-\dfrac{\phi_{m+1}}{(m+1)^2}\sin[(m+1)\overline z]\bigg]+\sum_{m=N+2}^\infty (-1)^{m+1}\dfrac{\phi_m}{m^2}\sin(m\overline z)\quad\forall \overline{z}\in\bigg(0,\frac{\pi}{4}\bigg).
\end{split}
\end{equation}
By Theorem \ref{monotonia} we know that $\phi_m>0$ and it strictly decreasing with respect to $m\in\mathbb{N}^+$ for all $y,w\in[-\ell,\ell]$; hence the following estimate holds:
\begin{equation}\label{phi1}
\begin{split}
\phi_1 \sin(\overline z)-\dfrac{\phi_2}{2^3}\sin(2\overline z)=\sin(\overline z)\bigg[\phi_1 -\dfrac{\phi_2}{2^2}\cos(\overline z)\bigg]>\dfrac{3}{4}\phi_2\sin(\overline z)>\dfrac{3}{4}\phi_{N+1}\sin(\overline z)\quad \forall \overline z \in\bigg(0,\frac{\pi}{2}\bigg)\,.
\end{split}
\end{equation}
Next, by exploiting the inequality 
\begin{equation*}\label{s}
\begin{split}
\dfrac{\sin(m\overline z)}{m^2}-\dfrac{\sin[(m+1)\overline z]}{(m+1)^{2}}>\sin(\overline z)\bigg[\dfrac{1}{m}-\dfrac{1}{m+1}\bigg]^2\quad \forall \overline z\in\bigg(0,\dfrac{\pi}{N+1}\bigg), \,\,\forall m=3,\dots,N,
\end{split}
\end{equation*}
(see Lemma \ref{lemmasin2} in the Appendix for a proof) and \eq{phi1}, we get 
\begin{equation}\label{dise4}
\begin{split}
\sum_{\substack{
		m=1\\
		\text{odd}}}^N&\bigg[ \dfrac{\phi_m}{m^2}\sin(m\overline z)-\dfrac{\phi_{m+1}}{(m+1)^2}\sin[(m+1)\overline z]\bigg]> \phi_{N+1}\sin(\overline z) \bigg[\dfrac{3}{4}+\sum_{\substack{
		m=3\\
		\text{odd}}}^N \bigg(\dfrac{1}{m}-\dfrac{1}{m+1}\bigg)^2\bigg]\,,
\end{split}
\end{equation}
for all $\overline z\in\bigg(0,\dfrac{\pi}{N+1}\bigg)$. On the other hand, through the monotonicity of the $\phi_m$, we get
\begin{equation}\label{dise5}
\bigg|\sum_{m=N+2}^\infty  (-1)^{m+1}\dfrac{\phi_m}{m^2}\sin(m\overline z)\bigg|\leq \phi_{N+1}\sum_{m=N+2}^\infty  \dfrac{1}{m^2}\quad \forall \overline z\in(0,\pi),\,\,\forall N\geq 3.
\end{equation}
From \eqref{dise4}-\eqref{dise5}, for all $N\geq 3$ odd, we infer 
\begin{equation*}\label{aim2}
\begin{split}
\mathcal{S}_1(\overline z)\geq 
\phi_{N+1}\,\bigg[\dfrac{3}{4}+\sum_{\substack{
		m=3\\
		\text{odd}}}^N &\bigg(\dfrac{1}{m}-\dfrac{1}{m+1}\bigg)^2\bigg]\,( \sin(\overline z)-\overline C_N) \quad \forall \overline z\in\bigg(0,\dfrac{\pi}{N+1}\bigg)\,,
\end{split}
\end{equation*}
where $$\overline C_N:=\left(\sum\limits_{m=N+2}^\infty  \dfrac{1}{m^2}\right)\left(\dfrac{3}{4}+\sum\limits_{\substack{
			m=3\\
			\text{odd}}}^N \bigg[\dfrac{1}{m}-\dfrac{1}{m+1}\bigg]^2\right)^{-1}\,. $$
Next we denote by $\overline z_N$ the unique solution to the equation 
\begin{equation*}\label{sin_ineq2}
\sin( z)=\overline C_N\qquad z\in(0,\pi/2).
\end{equation*} The above definition makes sense for all $N\geq 3$, odd, since the map $N\mapsto \overline C_N$ is positive, strictly decreasing and $0<\overline C_N<1$.

 We prove that
$
\overline C_N< \sin\big(\frac{\pi}{N+3}\big)$ for all $N\geq 3$, odd. To this aim we note that
$
\frac{3}{4}+\sum\limits_{\substack{
		m=3\\
		\text{odd}}}^N \big[\frac{1}{m}-\frac{1}{m+1}\big]^2> \frac{3}{4}$ and $\sum\limits_{m=N+2}^\infty  \frac{1}{m^2}< \int_{N+1}^\infty\frac{1}{x^2}\,dx=\frac{1}{N+1},
$
implying
$
\overline C_N< \frac{4}{3(N+1)}$ $\forall N\geq 3.$
Finally, by exploiting \eq{stimaseno}, we get $\overline C_N <\frac{4}{3(N+1)}\leq \frac{3}{N+3}\leq  \sin\big(\frac{\pi}{N+3}\big)$ for all $N\geq 3$.\par

Summarizing, from the above estimates we get
$$\mathcal{S}_1(\overline z)>0\qquad\forall \overline z\in\bigg[\dfrac{\pi}{N+3},\dfrac{\pi}{N+1}\bigg)\qquad \forall N\geq 3,\text{ odd}\,.$$
By repeating the above argument for all $3 \leq N \leq \overline N$ with $N$ and $\overline N$ odd and taking the union of the sets, we finally obtain
$
\mathcal{S}_1(z)>0$ $\forall z\in\big(\frac{3}{4}\pi,\pi-\frac{\pi}{\overline N+3}\big]$.
Hence, passing to the limit as  $\overline N\rightarrow+\infty$, and combining with \eqref{eq19} we obtain \eqref{S1}.
\end{proof}

	\begin{lemma}\label{lemma42}
	Let the series $\mathcal{S}_2(z)$ be as in \eqref{seriedef}; then
	\begin{equation}\label{S2}
	\mathcal{S}_2(z)<0\qquad\forall z \in (0,\pi).
	\end{equation}
\end{lemma}
\begin{proof}
It suffices noticing that:
$$\mathcal{S}_2(z)=\mathcal{S}_2(\pi-\overline z)=-\mathcal{S}_1(\overline z)\qquad \forall\,  z, \overline z\in (0, \pi)$$
with $\mathcal{S}_1$ as given in \eqref{seriedef}. Then, the thesis comes from Lemma \ref{lemma4} since $\mathcal{S}_1(\overline z)>0$ $\forall\,  \overline z\in (0, \pi)$.

\end{proof}

\textbf{Proof of Theorem \ref{greenmonotonia} completed.} 
We neglect the dependence on $y,w$ since it does not affect the results. Differentiating \eqref{green} with respect to $x$ we get
\begin{equation*}
\begin{split}
G_x(x,\rho)=\dfrac{1}{2\pi}\sum_{m=1}^{+\infty} \dfrac{\phi_m}{m^2}\sin(m\rho)\cos(mx)\,.
\end{split}
\end{equation*}
By \eqref{seriedef}, we may write
\begin{equation*}
\begin{split}
G_x(0,\rho)=\dfrac{\mathcal{S}_1(\rho)}{2\pi}\quad \text{and} \quad G_x(\pi,\rho)=\dfrac{\mathcal{S}_2(\rho)}{2\pi}\qquad \forall\rho\in(0,\pi).
\end{split}
\end{equation*}
 Therefore, the sign of $G_x(0,\rho)$ and of $G_x(\pi,\rho)$ follows from from Lemma \ref{lemma4} and Lemma \ref{lemma42}. \par

Next we turn to the sign of $G_x$ for $x=\frac{\pi}{2}$. Clearly, $G_x\big(\frac{\pi}{2},\frac{\pi}{2}\big)=0$. For $\rho \in (0,\frac{\pi}{2})$, we have
 $$G_x\bigg(\frac{\pi}{2},\rho\bigg)=\dfrac{1}{8\pi}\sum_{k=1}^\infty (-1)^{k}\dfrac{\phi_{2k}}{k^2}\sin(2k\rho)=\dfrac{1}{8\pi}\sum\limits_{k=1}^\infty (-1)^{k}\dfrac{\phi_{2k}}{k^2}\sin(k\widetilde\rho)\qquad \forall \widetilde \rho= 2 \rho \in (0,\pi)\,.$$
 For $\rho \in (\frac{\pi}{2}, \pi)$ we have
 $$G_x\bigg(\frac{\pi}{2},\rho\bigg)=\dfrac{1}{8\pi}\sum_{k=1}^\infty (-1)^{k}\dfrac{\phi_{2k}}{k^2}\sin(2k\rho)=\dfrac{1}{8\pi}\sum\limits_{k=1}^\infty \dfrac{\phi_{2k}}{k^2}\sin(k\widetilde\rho)\qquad \forall \widetilde \rho= 2 \rho-\pi \in (0,\pi).$$
Since the $\phi_{2k}$ are decreasing with respect to $k$, see Theorem \ref{monotonia}-\eqref{decreasing}, the proof of the sign of the above term follows by arguing as in the proof of Lemmas \ref{lemma4} and \ref{lemma42} with minor changes.

\section{Proof of Corollary \ref{hopfcor}, Corollary \ref{pppeig} and duality principle}\label{proofscorollaries}
\subsection{Proof of Corollary \ref{hopfcor}}
Since $u$ writes
$$  u(x,y)= \int_{\Omega} G(x,y,\rho,w) \,f(\rho,w)\,d\rho\, d w \quad \forall (x,y) \in \overline \Omega\,, $$
the proofs of both the sign and the monotonicity issues follow as direct consequence of the related properties of the Green function given in Proposition \ref{monotonia} and Theorem \ref{greenmonotonia}.
\subsection{Proof of Corollary \ref{pppeig}}
	The proof that the first eigenfunction is of one sign, and hence simple, follows by exploiting the so-called dual cone decomposition technique which relies on the positivity preserving property stated in Corollary \ref{hopfcor}. Since the proof is standard we omit it and we refer the interested reader to \cite[Lemma 7.2]{befafega} where the same issue, together with the simplicity of the first eigenvalue, was proved for a related fourth order eigenvalue problem in dimension 1. As concerns the regularity of $u_p$, it follows by combining elliptic regularity and embedding arguments. Indeed, it was proved in \cite[Lemma 4.2]{fergaz} that the operator in \eq{weight} satisfies the the complementing conditions of Agmon-Douglis-Nirenberg \cite{ADN}, hence elliptic regularity theory applies. In particular, from $\lambda_1(p) p u_p\in L^{\infty}(\Omega)$ we infer that $u_p \in W^{4,q}(\Omega)$ for all $1<q< +\infty$; then the thesis comes by noticing that $W^{4,q}(\Omega)\subset C^{3,\gamma}(\overline \Omega)$ for some $0<\gamma <1$, see \cite[Theorem 5.4]{adams}. \par
	Now we turn to the sign of $(u_p)_x$ on the short edges. Recalling Theorem \ref{greenmonotonia}, we have
	$$  (u_p)_x(0,y)=\lambda_1(p)  \int_{\Omega} G_x(0,y,\rho,w) \,p(\rho,w)\,u_p(\rho,w)\,d\rho\, d w>0 \quad \forall y \in [-\ell, \ell]\,. $$
	Similarly, we get $ (u_p)_x(\pi,y)<0$ for all $y \in [-\ell, \ell] $ and this concludes the proof. 
\subsection{Duality principle}

Let $\cG: L^{2}(\Omega) \to H^2_*(\Omega)$ denote the solution operator for the biharmonic equation under partially hinged boundary conditions defined by
\begin{equation}\label{green-operator1}
(\cG f,\varphi)_{H^2_*(\Omega)}:= \int_\Omega f\varphi \qquad \text{for all } \varphi\in H^2_*(\Omega).
\end{equation}
In terms of the Green function \eqref{green}, we get the usual integral representation:
\begin{equation}\label{green-operator}
[\cG f](x,y)= \int_\Omega G(x,y,\rho,w)f(\rho,w)\,d\rho dw \qquad \forall (x,y) \in \Omega.
\end{equation}
 Next, inspired by \cite[Lemma 12]{begawe}, we associate to  \eqref{lambdaP} the following dual maximization problem: 
\begin{equation}
\label{eq:12} \Theta_1(p )= \sup_{v \in L^{2}(\Omega)\setminus\{0\}}\frac{\int_{\Omega}\cG(p\,v)\,p\,v\:dxdy}{\|\sqrt{p}\,v\|^2_2}\,.
\end{equation}
By standard compactness arguments, the supremum in \eqref{eq:12} is achieved, furthermore, if $v$ is a maximizer, by exploiting \eqref{green-operator} and the positivity of $G$, also $|v|\in L^{2}(\Omega)$ is a maximizer; hence, a maximizer to \eqref{eq:12} can always be assumed nonnegative.  Finally we state:
\begin{lemma}\label{dualityprinciple} (Duality principle) Let $p\in P_{\alpha, \beta}$, then $\Theta_1(p)= \lambda_1^{-1}(p)$. Furthermore,
\begin{itemize}
	\item[$(i)$]  if $u \in H^2_*(\Omega)$ is a positive minimizer of \eqref{lambdaP} with $\|\sqrt{p}\,u\|_2=1$, then $u$ is a maximizer for \eqref{eq:12};
\item[$(ii)$]  if $v \in L^{2}(\Omega)$ is a nonnegative maximizer of \eqref{eq:12} with $\|\sqrt{p}\,v\|_2=1$, then $v \in H^2_*(\Omega)$ and it is a minimizer for \eqref{lambdaP}, hence positive. 
	\end{itemize}
\end{lemma}\begin{proof}
	Let $u \in H^2_*(\Omega)$ be a positive minimizer for \eqref{lambdaP} with $\|\sqrt{p}\,u\|_2=1$. Then $u$ solves problem \eqref{weight}, therefore $u=\lambda_1 \cG (p \,u)$. By multiplying both sides of this equality by $p\,u$
	and integrating over $\Omega$, we get
	$
	\lambda_1\int_\Omega\cG(p \,u)\,p \,u\:dxdy=\int_\Omega p \,u^2\,dx=\|\sqrt{p}\,u\|_2^2=1,
	$
	hence
	\begin{equation}
	\label{upper-ineq} \Theta_1 \geq \frac{\int_{\Omega}\cG(p\,u)\,p\,u\:dxdy}{\|\sqrt{p}u\|^2_2}=\frac{1}{\lambda_1}.
	\end{equation}
	Viceversa let $v \in L^{2}(\Omega)$ be a nonnegative maximizer for \eqref{eq:12} with $\|\sqrt{p}\,v\|_2=1$. The corresponding Euler-Lagrange equation in weak form reads
	$$
	\int_{\Omega}\cG(p\, v)\,p\,\varphi\:dxdy=\Theta_1
	\int_{\Omega}p\,v\,\varphi\,dxdy\qquad\forall \varphi \in L^{2}(\Omega),
	$$
	implying $\cG(p\, v)=\Theta_1\,v$ a. e. in $\Omega$.
	Therefore, taking $v=\frac{1}{\Theta_1}\cG(p\, v) \in H^2_*(\Omega)$, we obtain by
	(\ref{green-operator1})
	$$
	\Theta_1 \|v\|_{H^2_*(\Omega)}^{2}=\Theta_1 (v, v )_{H^2_*(\Omega)} = (\cG(p\, v),v )_{H^2_*(\Omega)}=\int_{\Omega}p\, v^{2}\,dxdy=1,
	$$
	so that
	\begin{equation}
	\label{lower-ineq}
	\lambda_1 \leq \frac{\|v\|_{H^2_*(\Omega)}^{2}}{\|\sqrt{p}\,v\|_2^2}=\frac{1}{\Theta_1}.
	\end{equation}
By (\ref{upper-ineq}) and (\ref{lower-ineq}) we get $\Theta_1(p)= \lambda_1^{-1}(p)$. But then the first inequality in (\ref{upper-ineq}) must be an equality, and $(i)$ follows. Similarly, the first
	inequality in (\ref{lower-ineq}) must be an equality, and $(ii)$ follows.
\end{proof}

\section{Proof of Theorem \ref{partialsym}} \label{6}
By Proposition \ref{thm-exist-qual} we know that there exists an optimal pair $(p_{\widehat{u}},{\widehat u}) \in P_{\alpha,\beta} \times H^2_*(\Omega)$, with $
p_{\widehat{u}}$ as given in \eqref{pS}. For the sake of simplicity, in the following we will simply denote $(p_u, u)$ this pair, hence
	\begin{equation}\label{pu}
p_u (x,y):= \alpha \chi_{\{0<u \leq \sqrt{t}\}}(x,y)+ \beta \chi_{\{u > \sqrt{t}\}}(x,y)\,\quad \text{with }(x,y)\in  \Omega \,
\end{equation}
and $t>0$ is fixed as in the statement of Proposition \ref{thm-exist-qual}.\par

To begin with we recall some notations about the polarization of a function, adapting the technique to our framework. Let $\mathcal{H}_{\frac{\pi}{2}}\subset \mathbb{R}^2$ be the half-plane defined in \eqref{iperp}; for every $(x,y)\in\mathbb{R}^2$ we denote by $(\overline x, y)\in \mathbb{R}^2$ the reflection of $(x,y)$ with respect to $\partial\mathcal{H}_{\frac{\pi}{2}}$, i.e. $\overline x=\pi-x$. For every measurable function $v: \Omega\rightarrow\mathbb{R}$ we define its polarization with respect to $\mathcal{H}_{\frac{\pi}{2}}$ as $v_{\mathcal{H}_{\frac{\pi}{2}}}: \Omega \rightarrow\mathbb{R}$ such that
\begin{equation*}
	v_{\mathcal{H}_{\frac{\pi}{2}}}(x,y):=\begin{cases}
	\max\{v(x,y),v(\overline x,y)\}\qquad &\text{if }(x,y)\in\mathcal{H}_{\frac{\pi}{2}}\cap \Omega\\\min\{v(x,y),v(\overline x,y)\}\qquad &\text{if }(x,y)\in \Omega\setminus \mathcal{H}_{\frac{\pi}{2}} \,.
	\end{cases}
\end{equation*}
In the sequel, for the sake of brevity, we will simply write $\mathcal{H}$ instead of $\mathcal{H}_{\frac{\pi}{2}}$ and $v_{\mathcal{H}}$ instead of $v_{\mathcal{H}_{\frac{\pi}{2}}}$. It is readily seen that the following pointwise identity holds:
\begin{equation}\label{prop2}
v(x,y)+v(\overline x,y)=v_{\mathcal{H}}(x,y)+v_{\mathcal{H}}(\overline x,y)\qquad \forall (x,y)\in \Omega.
\end{equation}
Now, for $t>0$ as fixed in \eqref{pu}, we also consider the weight:
\begin{equation*}
p_{u_{\mathcal{H}}} (x,y):= \alpha \chi_{\{0<u_{\mathcal{H}} \leq \sqrt{t}\}}(x,y)+ \beta \chi_{\{u_{\mathcal{H}} > \sqrt{t}\}}(x,y)\, \quad \text{with }(x,y)\in  \Omega\,.
\end{equation*}
By direct inspection, arguing as in the proof of  Lemma \ref{lemma5} below, we have that $\int_{\Omega}p_{u_{\mathcal{H}}}\,dxdy=\int_{\Omega}p\,dxdy=|\Omega|$, hence $p_{u_{\mathcal{H}}}\in P_{\alpha, \beta}$. Next we state two technical lemmas whose proofs can be obtained by slightly modifying the proofs of similar statements in \cite{CV}, hence we omit them.
\begin{lemma}\label{lemma00} \cite[Lemma 5.3]{CV} The following identity holds:
	\begin{equation*}
		[p_u\,u]_\mathcal{H}\equiv p_{u_{\mathcal{H}}}u_\mathcal{H}\, \text{ in } \Omega\,.
	\end{equation*}
\end{lemma}
\begin{lemma}\label{lemma10} \cite[Lemma 5.4]{CV}
	There holds:
	\begin{itemize}
	\item[$(i)$] 	if
	$ p_{u}(x,y)u(x,y)\equiv [p_u(x,y)\,u(x,y)]_\mathcal{H}\text{ in } \Omega$, then $u(x,y)\equiv u_\mathcal{H}(x,y)\text{ in }  \Omega$;
	\item[$(ii)$] 	 if
	$ p_{u}(\overline x,y)u(\overline x,y)\equiv [p_u(x,y)\,u(x,y)]_\mathcal{H}\text{ in }  \Omega$, then $u(\overline x,y)\equiv u_\mathcal{H}(x,y)\text{ in } \Omega$\,.
	\end{itemize}	
\end{lemma}
%
%
%
%
%

Finally, we prove the identity:

\begin{lemma}\label{lemma5}
	We have
	\begin{equation*}
		\int_{\Omega}p_{u_\HH}u^2_\HH\,dxdy=\int_{\Omega} p_u u^2\,dxdy.
	\end{equation*} 
	
\end{lemma}
\begin{proof}
	We compute:
	{\small \begin{equation*}
	\begin{split}
	&\int_{\Omega}p_{u_\HH}(x,y)u^2_\HH(x,y)\,dxdy=\alpha\int_{\{u_\HH(x,y)\leq\sqrt{t}\}}u^2_\HH(x,y)\,dxdy+\beta\int_{\{u_\HH(x,y)\geq\sqrt{t}\}}u^2_\HH(x,y)\,dxdy\\
	&=\alpha\int_{\{(x,y)\in \Omega \cap \HH:\,u(x,y)\geq u(\overline x,y)\,\cap\, u(x,y)\leq\sqrt{t}\}}u^2(x,y)\,dxdy+\alpha\int_{\{(x,y)\in  \Omega \cap \HH:\,u(\overline x,y)> u(x,y)\,\cap\, u(\overline x,y)\leq\sqrt{t}\}}u^2(\overline x,y)\,dxdy\\
	&+\alpha\int_{\{(x,y)\in\Omega\setminus \HH:\,u(x,y)> u(\overline x,y)\,\cap\, u(\overline x,y)\leq\sqrt{t}\}}u^2(\overline x,y)\,dxdy+\alpha\int_{\{(x,y)\in\Omega\setminus\HH:\,u(\overline x,y)\geq u(x,y)\,\cap\, u( x,y)\leq\sqrt{t}\}}u^2(x,y)\,dxdy\\
	&+\beta\int_{\{(x,y)\in  \Omega \cap \HH:\,u(x,y)\geq u(\overline x,y)\,\cap\, u(x,y)>\sqrt{t}\}}u^2(x,y)\,dxdy+\beta\int_{\{(x,y)\in \Omega \cap \HH:\,u(\overline x,y)> u(x,y)\,\cap\, u(\overline x,y)>\sqrt{t}\}}u^2(\overline x,y)\,dxdy\\
	&+\beta\int_{\{(x,y)\in\Omega\setminus\HH:\,u(x,y)> u(\overline x,y)\,\cap\, u(\overline x,y)>\sqrt{t}\}}u^2(\overline x,y)\,dxdy+\beta\int_{\{(x,y)\in\Omega\setminus\HH:\,u(\overline x,y)\geq u(x,y)\,\cap\, u( x,y)>\sqrt{t}\}}u^2(x,y)\,dxdy\\
	&=\alpha\int_{\{(x,y)\in\Omega: u(x,y)\leq\sqrt{t}\}}u^2(x,y)\,dxdy+\beta\int_{\{(x,y)\in\Omega: u(x,y)>\sqrt{t}\}}u^2(x,y)\,dxdy=\int_{\Omega}p_{u}(x,y)u^2(x,y)\,dxdy\,,
	\end{split}
	\end{equation*}}	
	where the first equality in the last line simply comes by changing variables.
\end{proof}

Next we turn to the proof of some reflection properties of the Green function $G$ that will be crucial in the following.
\begin{lemma}\label{lemma0}
	For $(x,y)\in \mathcal{H} \cap \overline \Omega$ and $(\rho,w)\in \mathcal{H} \cap \overline \Omega$ we have:
	\begin{itemize}
		\item[$(i)$] $G(x,y,\rho,w)\geq \max\{G(\overline x,y,\rho,w),G(x,y,\overline\rho,w)\}$ with strict inequality if $x,\rho \neq 0,\frac{\pi}{2}$;
		\item[$(ii)$] $G(x,y,\rho,w)=G(\overline x,y,\overline\rho,w)$;
		\item[$(iii)$] $G(\overline x,y,\rho,w)=G(x,y,\overline\rho,w)$.
		\end{itemize}
	\end{lemma}
	
	\begin{proof}
	
	The variables $y,w$ do not play a role, hence we fix them and we do not write them in the proof.

\vspace{2mm}
\emph{Proof of (i).} We study the sign of
\begin{equation*}
\begin{split}
G(x,\rho)-G(\overline x,\rho)&=\dfrac{1}{\pi}\sum_{m=1}^{+\infty} \dfrac{\phi_m}{m^3}\sin(m\rho)\sin\bigg[m\bigg(x-\frac{\pi}{2}\bigg)\bigg]\cos\bigg(m\frac{\pi}{2}\bigg)\\
		&=\dfrac{1}{\pi}\sum_{k=1}^{+\infty} \dfrac{\phi_{2k}}{(2k)^3}\sin(2k\rho)(-1)^k\sin(2kx)(-1)^{k}\qquad \forall (x,\rho)\in[0,\pi/2]^2.
\end{split}
\end{equation*}For $\widetilde x=2x$ and $\widetilde \rho=2\rho$, we observe that
\begin{equation*}
\dfrac{1}{8\pi}\sum_{k=1}^{+\infty} \dfrac{\phi_{2k}}{k^3}\sin(k\widetilde \rho)\sin(k\widetilde x)>0 \qquad \forall (\widetilde x,\widetilde \rho)\in(0,\pi)^2\,.
\end{equation*}
The proof the above inequality can be obtained by repeating, with minor changes, the proof of \eqref{tesi} as given in \cite[Theorem 2.2]{ppp}. The main ingredient is the monotonicity of the functions $\phi_{2k}$ with respect to $k$. This implies that $G(x,\rho)-G(\overline x,\rho)>0$ for all $(x,\rho)\in (0,\pi/2)^2$. We observe that for $x\in\{0,\pi/2\}$ or $\rho\in\{0,\pi/2\}$ we have $G(x,\rho)=G(\overline x,\rho)$, giving  $G(x,\rho)\geq G(\overline x,\rho)$ for all $(x,\rho)\in [0,\pi/2]^2$.

Repeating the above arguments, but inverting the variables $x$ and $\rho$, we get the statement $(i)$.\par 
\vspace{2mm}
\emph{Proof of (ii) and (iii).} 
For all $(x,\rho)\in [0,\pi]^2$ we get 
\begin{equation*}
\begin{split}
&G(\overline x,\overline \rho)=
\dfrac{1}{2\pi}\sum\limits_{m=1}^{+\infty} \dfrac{\phi_m}{m^3}\sin[m(\pi-\rho)]\sin[m(\pi-x)]=\dfrac{1}{2\pi}\sum\limits_{m=1}^{+\infty} \dfrac{\phi_m}{m^3}\sin(m\rho)\sin(mx)=G(x,\rho)\,,\\
 &G(\overline x,\rho)=
 \dfrac{1}{2\pi}\sum\limits_{m=1}^{+\infty} \dfrac{\phi_m}{m^3}\sin(m\rho)\sin[m(\pi-x)]=\dfrac{1}{2\pi}\sum\limits_{m=1}^{+\infty}(-1)^{m+1} \dfrac{\phi_m}{m^3}\sin(m\rho)\sin(mx)=G(x,\overline \rho)\,.
 \end{split}
 \end{equation*}

\end{proof}

Thanks to Lemma \ref{lemma0} we obtain:

\begin{lemma}
	\label{sec:polar-symm-prop}
	Let $\cG: L^{2}(\Omega) \to H^2_*(\Omega)$ denote the solution operator defined by \eqref{green-operator1}. Then,
	\begin{equation}
	\label{eq:27}
	\begin{split}
	&\int_{\Omega } \cG(p_u u)\, \,p_u(x,y)u(x,y)\:dxdy\leq	\int_{\Omega } \cG(p_{u_\HH}u_\HH)\,\,p_{u_\HH}(x,y)u_\HH(x,y)\:dxdy,
	\end{split}
	\end{equation}
	and the equality holds in \eqref{eq:27} if and only if 
	$$p_u(x,y)u(x,y)=[p_{u}(x,y)u(x,y)]_\HH \text{ a.e. in } \Omega\quad  \text{ or } \quad p_u(\overline x,y)u(\overline x,y)= [p_u(x,y)u(x,y)]_\HH \text{ a.e. in } \Omega\,.$$
\end{lemma}
\begin{proof} We define
	$$
	A(g,h):=\int_{\Omega \times \Omega} G(x,y,\rho,w)
	p_g (x,y) p_h(\rho,w) g(x,y)h(\rho,w)\:dxdyd\rho dw 
	$$
	where $g$ and $h$ have to be meant equal to $u$ or $u_\HH$\,. Then, by writing $\Omega \times \Omega=[(\Omega \cap \HH) \times(\Omega \cap \HH)] \cup [(\Omega \cap \HH) \times(\Omega \setminus \HH)] \cup [(\Omega \setminus \HH) \times(\Omega \cap \HH)]\cup [(\Omega \setminus \HH) \times(\Omega \setminus \HH)]$ and changing variables properly, we get
	\begin{align*}
	&A(u_\HH,u_\HH)-A(u_\HH,u)=\nonumber\\
	&\!\int_{(\Omega \cap \HH) \times(\Omega \cap \HH)}\! G(x,y,\rho,w)p_{u_\HH}(x,y)u_\HH(x,y)[p_{u_\HH}(\rho,w)u_\HH(\rho,w)-p_u(\rho,w)u(\rho,w)]\:dxdyd\rho dw \nonumber\\
	+&\!\int_{(\Omega \cap \HH) \times(\Omega \cap \HH)}\! G(x,y,\overline \rho,w)p_{u_\HH}(x,y)u_\HH(x,y)[p_{u_\HH}(\overline \rho,w)u_\HH(\overline \rho,w)-p_u(\overline \rho,w)u(\overline \rho,w)]\:dxdyd\rho dw \nonumber\\
	+&\!\int_{(\Omega \cap \HH) \times(\Omega \cap \HH)}\! G(\overline x,y,\rho,w)p_{u_\HH}(\overline x,y)u_\HH(\overline x,y)[p_{u_\HH}(\rho,w)u_\HH(\rho,w)-p_u(\rho,w)u(\rho,w)]\:dxdyd\rho dw \nonumber\\
	+&\!\int_{(\Omega \cap \HH) \times(\Omega \cap \HH)}\! G(\overline x,y,\overline \rho,w)p_{u_\HH}(\overline x,y)u_\HH(\overline x,y)[p_{u_\HH}(\overline \rho,w)u_\HH(\overline \rho,w)-p_u(\overline \rho,w)u(\overline \rho,w)]\:dxdyd\rho dw.
	\label{A1}
	\end{align*}
By Lemma \ref{lemma00} and \eqref{prop2} we have 
$$
p_{u_\HH}(\rho,w)u_\HH(\rho,w)-p_u(\rho,w)u(\rho,w)=-[p_{u_\HH}(\overline\rho,w)u_\HH(\overline\rho,w)-p_u(\overline\rho,w)u(\overline\rho,w)]\qquad \forall (\rho,w)\in \Omega,
$$ 
so that 
\small
\begin{equation*}
\begin{split}
&A(u_\HH,u_\HH)-A(u_\HH,u)=\nonumber\\
&\!\int_{(\Omega \cap \HH) \times(\Omega \cap \HH)}\! [G(x,y,\rho,w)-G(x,y,\overline\rho,w)]p_{u_\HH}(x,y)u_\HH(x,y)[p_{u_\HH}(\rho,w)u_\HH(\rho,w)-p_u(\rho,w)u(\rho,w)]\:dxdyd\rho dw \nonumber\\
+&\!\int_{(\Omega \cap \HH) \times(\Omega \cap \HH)}\! [G(\overline x,y,\rho,w)-G(\overline x,y,\overline\rho,w)]p_{u_\HH}(\overline x,y)u_\HH(\overline x,y)[p_{u_\HH}(\rho,w)u_\HH(\rho,w)-p_u(\rho,w)u(\rho,w)]\:dxdyd\rho dw.
\end{split}
\label{A2}
\end{equation*}
\normalsize
Then, thanks to Lemma \ref{lemma0} $(ii)$ and $(iii)$ we conclude that
\begin{equation}
\begin{split}
&A(u_\HH,u_\HH)-A(u_\HH,u)=\!\int_{(\Omega \cap \HH) \times(\Omega \cap \HH)}\! [G(x,y,\rho,w)-G(x,y,\overline\rho,w)]\\
&\times [p_{u_\HH}(x,y)u_\HH(x,y)-p_{u_\HH}(\overline x,y)u_\HH(\overline x,y)][p_{u_\HH}(\rho,w)u_\HH(\rho,w)-p_u(\rho,w)u(\rho,w)]\:dxdyd\rho dw\,.
\end{split}
\label{A3}
\end{equation}
With similar arguments we get
\begin{equation}
\begin{split}
&A(u_\HH,u)-A(u,u)=\!\int_{(\Omega \cap \HH) \times(\Omega \cap \HH)}\! [G(x,y,\rho,w)-G(x,y,\overline\rho,w)]\\&\times [p_{u}(x,y)u(x,y)-p_{u}(\overline x,y)u(\overline x,y)][p_{u_\HH}(\rho,w)u_\HH(\rho,w)-p_u(\rho,w)u(\rho,w)]\:dxdyd\rho dw
\end{split}
\label{A4}
\end{equation}
and combining \eqref{A3}-\eqref{A4} we obtain
\begin{equation*}
\begin{split}
&A(u_\HH,u_\HH)-A(u,u)=\!\int_{(\Omega \cap \HH) \times(\Omega \cap \HH)}\! [G(x,y,\rho,w)-G(x,y,\overline\rho,w)] [p_{u_\HH}(\rho,w)u_\HH(\rho,w)-p_u(\rho,w)u(\rho,w)]\\
&\times[p_{u_\HH}(x,y)u_\HH(x,y)-p_{u_\HH}(\overline x,y)u_\HH(\overline x,y)+p_{u}(x,y)u(x,y)-p_{u}(\overline x,y)u(\overline x,y)]\:dxdyd\rho dw.
\end{split}
\end{equation*}
Now, by Lemma \ref{lemma0}-$(i)$, we know that $G(x,y,\rho,w)-G(x,y,\overline\rho,w)\geq0$ while, by Lemma \ref{lemma00}, we get
$$p_{u_\HH}(\rho,w)u_\HH(\rho,w)-p_u(\rho,w)u(\rho,w)=[p_u(\rho,w)\,u(\rho,w)]_\mathcal{H}-p_u(\rho,w)u(\rho,w)\geq 0\qquad \forall (\rho,w)\in \Omega \cap \HH\,.$$
Finally, \eqref{eq:27} follows by noticing that, through Lemma \ref{lemma00} and \eqref{prop2}, we have
\begin{equation*}
\begin{split}
&[p_{u_\HH}(x,y)u_\HH(x,y)-p_{u_\HH}(\overline x,y)u_\HH(\overline x,y)+p_{u}(x,y)u(x,y)-p_{u}(\overline x,y)u(\overline x,y)]\\
&=2\{[p_u(x,y)u(x,y)]_\HH-p_{u}(\overline x,y)u(\overline x,y)\}\geq 0\qquad \forall (x,y)\in \Omega \cap\HH\,.
\end{split}
\end{equation*}
To prove the last part of the statement we set
$D_1:=\{(x,y)\in\Omega \cap \HH:\, p_u(x,y)u(x,y)>p_u(\overline x,y)u(\overline x,y)\}$ and $D_2:=\{(\rho,w)\in\Omega \cap \HH:\,  [p_u(\rho,w)u(\rho,w)]_{\HH}>p_u(\rho,w)u(\rho,w)\}$. If equality holds in \eqref{eq:27} we get
\begin{equation}
\begin{split}
0&=A(u_\HH,u_\HH)-A(u,u) \\
&= \!\int_{D_1\times D_2}\! [G(x,y,\rho,w)-G(x,y,\overline\rho,w)] \{[p_u(\rho,w)\,u(\rho,w)]_\mathcal{H}-p_u(\rho,w)u(\rho,w)\}\\&\times2\{[p_u(x,y)u(x,y)]_\HH-p_{u}(\overline x,y)u(\overline x,y)\}\:dxdyd\rho dw\,.
\end{split}
\label{A6}
\end{equation}
Now, \eqref{A6} makes sense if and only if $|D_1|=0$ or $|D_2|=0$, i.e., if and only if $[p_u(x,y)u(x,y)]_\HH=p_u(\overline x,y)u(\overline x,y)$ or $[p_{u}(x,y)u(x,y)]_\HH=p_u(x,y)u(x,y)$
a.e.\ in $\Omega$.
\end{proof}
\par \medskip \par
\textbf{Proof of Theorem \ref{partialsym} completed.} Thanks to Lemma \ref{dualityprinciple} we have that $u$ is a maximizer for \eqref{eq:12} with $p=p_u$. Then, since $(u,p_u)$ is an optimal pair, $u_\HH \in L^2(\Omega)$ and ${p_{u_{\mathcal{H}}}}\in P_{\alpha, \beta}$, we infer that
	$$
	\frac{\int_{\Omega}\cG(p_uu)\,p_u\,u\:dxdy}{\|\sqrt{p_u}\,u\|^2_2}=\Theta_1(p_u)\geq \Theta_1(p_{u_\HH}) \geq\frac{\int_{\Omega}\cG(p_{u_\HH}u_\HH)\,p_{u_\HH}\,u_\HH\:dxdy}{\|\sqrt{p_{u_\HH}}\,u_\HH\|^2_2}.
	$$
	Recalling that, by Lemma \ref{lemma5}, $\|\sqrt{p_{u_\HH}}u_\HH\|_2=\|\sqrt{p_u} u\|_2$, from above we get that
	$$\int_{\Omega}\cG(p_uu)\,p_u\,u\:dxdy \geq \int_{\Omega}\cG(p_{u_\HH}u_\HH)\,p_{u_\HH}\,u_\HH\:dxdy\,.$$
	
	Then, by Lemma \ref{sec:polar-symm-prop}, \eqref{eq:27} holds with the equality and, in view of Lemma \ref{lemma10}, this implies $u(x,y)=u_\HH(x,y)$ or $u(\overline x,y)= u_\HH(x,y)$ a.e. in $\Omega$. Since $u$ is continuous, we obtain
	\begin{equation}\label{eq10}
	\begin{split}
u(x,y)\geq u(\overline x,y)\quad \text{in } \overline \Omega\cap \HH \qquad	\text{or} \qquad u(x,y)\leq u(\overline x,y)\quad \text{in }\overline \Omega\cap \HH.
	\end{split}
	\end{equation}
	Let us consider the first case of \eqref{eq10}; then, it is readily seen that:
	\begin{equation}\label{eqiii1}
	p_{u}(x,y)u(x,y)\geq p_u(\overline x,y)\,u(\overline x,y)\quad \forall (x,y)\in \overline \Omega\cap \HH\,.
	\end{equation}
	Indeed, if $p_{u}(x_0,y_0)u(x_0,y_0)< p_u(\overline x_0,y_0)\,u(\overline x_0,y_0)$ for some $(x_0,y_0)\in \overline \Omega\cap \HH$, by \eqref{eq10} we get $p_{u}(x_0,y_0)=\alpha$ and $p_{u}(\overline x_0,y_0)=\beta$. But then $u(\overline x_0,y_0)\leq u( x_0,y_0) \leq \sqrt{t}$ and $p_{u}(\overline x_0,y_0)=\alpha$ which is a contradiction.\par Suppose now that there exists $(x_1,y_1)\in  \overline \Omega\cap \HH$ such that the strict inequality holds in the first of \eqref{eq10}, clearly $x_1 \neq 0, \pi/2$. Then, by continuity, there exists a subset $U\subset( \overline \Omega\cap \HH)$ of positive measure such that $u(x,y)>u(\overline x,y)$ for all $(x,y)\in U$ and, by arguing as for the proof of \eqref{eqiii1}, such that
	\begin{equation}\label{eqiii}
	p_{u}(x,y)u(x,y)> p_u(\overline x,y)\,u(\overline x,y)\quad \forall (x,y)\in U\,.
	\end{equation}
  Finally, through Lemma \ref{lemma0}, \eqref{eqiii1} and \eqref{eqiii}, for all $(x,y) \in (0, \frac{\pi}{2})\times[-\ell,\ell]$ we obtain
	\small
	\begin{equation*}
	\begin{split}
	&u(x,y)-u(\overline x,y)=\int_{\Omega}[G(x,y,\rho,w)-G(\overline x, y,\rho,w)]p_u(\rho,w)u(\rho,w)\,d\rho dw\\
	&=\int_{\Omega\cap \HH}\{
	[G(x,y,\rho,w)-G(\overline x, y,\rho,w)]p_u(\rho,w)u(\rho,w)+[G(x,y,\overline \rho,w)-G(\overline x, y,\overline \rho,w)]p_u(\overline \rho,w)u(\overline \rho,w)\}\,d\rho dw\\
	&=\int_{\Omega\cap \HH}
	[G(x,y,\rho,w)-G(\overline x, y,\rho,w)][p_u(\rho,w)u(\rho,w)-p_u(\overline \rho,w)u(\overline \rho,w)]\,d\rho dw\\
	&\geq \int_{U}
	[G(x,y,\rho,w)-G(\overline x, y,\rho,w)][p_u(\rho,w)u(\rho,w)-p_u(\overline \rho,w)u(\overline \rho,w)]\,d\rho dw>0,
	\end{split}
	\end{equation*}
	\normalsize
	implying that $(i)$ or $(iii)$ holds. Similarly, if we consider the second inequality in \eqref{eq10}, we get that $(ii)$ or $(iii)$ holds. This concludes the proof.

\section{Proof of Proposition \ref{lemma2}}
First, for all $ \rho\in(0,\pi)$ and $y,w\in[-\ell,\ell]$, we note that
	$$
	G_x\bigg(\frac{\pi}{2},y,\rho,w\bigg)=\dfrac{1}{8\pi}\sum\limits_{k=1}^\infty (-1)^{k}\dfrac{\phi_{2k}(y,w)}{k^2}\sin(2k\rho)=-G_x\bigg(\frac{\pi}{2},y,\overline\rho,w\bigg) \,.
	$$
By exploiting the above equality we write
\begin{equation*}
\begin{split}
u_x\bigg(\frac{\pi}{2},y\bigg)=\int_{-\ell}^\ell \int_0^{\pi/2} G_x\bigg(\frac{\pi}{2},y,\rho,w\bigg)[p_u(\rho,w)u(\rho,w)-p_u(\overline\rho,w)u(\overline \rho,w)]\,d\rho dw\quad \forall y\in[-\ell,\ell].
\end{split}
\end{equation*}
From Theorem \ref{greenmonotonia} we know that $G_x\big(\frac{\pi}{2},y,\rho,w\big)<0$ for all $\rho\in \big(0,\frac{\pi}{2}\big)$; then, if case $(i)$ of Theorem \ref{partialsym} holds, by \eqref{eqiii1}-\eqref{eqiii}, we get $p_u(\rho,w)u(\rho,w)>p_u(\overline\rho,w)u(\overline \rho,w)$ and, in turn, that $u_x\big(\frac{\pi}{2},y\big)<0$ for all $y\in[-\ell,\ell]$. Similarly, the reverse inequality holds if case $(ii)$ occurs. Finally, when $(iii)$ holds, then $p_u(\rho,w)u(\rho,w)\equiv p_u(\overline\rho,w)u(\overline \rho,w)$ in $\overline \Omega$, hence $u_x\big(\frac{\pi}{2},y\big)=0$ for all $y\in[-\ell,\ell]$.

\section*{Appendix}
\begin{lemma}\label{lemmasin2}
	Let $N\geq 3$ be an integer. For all $z\in\big(0,\frac{\pi}{N+1}\big)$ and for all $m=3,\dots,N$, there holds
	$$
	\upsilon_m(z):=\dfrac{\sin(mz)}{m^2}-\dfrac{\sin[(m+1)z]}{(m+1)^{2}}-\sin(z)\bigg[\dfrac{1}{m}-\dfrac{1}{(m+1)}\bigg]^2>0\,.
	$$
\end{lemma}

\begin{proof}
 Clearly, $\upsilon_m(0)=0$; we set $a_m:=\big[\frac{1}{m^2}-\frac{1}{(m+1)^2}\big]^2$ and we compute 
	$$
	\upsilon_m'(z)=\frac{\cos(m z)}{m}-\frac{\cos[(m+1) z]}{m+1}-a_m\cos(z)\qquad\upsilon_m''(z)=\sin[(m+1)z]-\sin(mz)+a_m\sin(z).
	$$
Using the complex identities for the trigonometric functions we obtain
	\begin{equation}\label{cx}
	\begin{split}
	\sin[(m+1)\overline z]-\sin(m\overline z)=0\quad &\iff\quad \overline z=2k\pi,\frac{(1+2k)\pi}{2m+1}\quad \forall k\in\mathbb{Z}
	\end{split}
	\end{equation}
	Hence $\sin[(m+1) z]>\sin(m z)$ for $z\in\big(0,\frac{\pi}{2m+1}\big)$ and $\upsilon_m''(z)>0$ for $z\in\big(0,\frac{\pi}{2m+1}\big)$; this readily implies that $\upsilon_m(z)>0$ for $z\in\big(0,\frac{\pi}{2m+1}\big]$.  \par 
	
	For $z\in\big(\frac{\pi}{2m+1},\frac{\pi}{m+1}\big) $ we have
	\begin{equation*}
	\begin{split}
	\upsilon_m(z)&=\frac{\sin(m z)}{m^2}-\frac{\sin[(m+1) z]}{(m+1)^{2}}-\dfrac{\sin z}{m}\bigg[\dfrac{1}{m}-\dfrac{1}{m+1}\bigg]+\dfrac{\sin z}{m+1}\bigg[\dfrac{1}{m}-\dfrac{1}{m+1}\bigg]\\&>\frac{\sin(m z)}{m^2}-\frac{\sin[(m+1) z]}{(m+1)^{2}}-\dfrac{\sin\big[\frac{\pi}{m+1}\big] }{m}\bigg[\dfrac{1}{m}-\dfrac{1}{m+1}\bigg]:=\overline \upsilon_m (z).
	\end{split}
	\end{equation*}
We study the sign of $\overline \upsilon_m (z)$ for $z\in\big(\frac{\pi}{2m+1},\frac{\pi}{m+1}\big)$. We have $$\overline \upsilon_m'' (z)=\sin[(m+1) z]-\sin(m z)<0\qquad \forall z\in\bigg(\frac{\pi}{2m+1},\frac{\pi}{m+1}\bigg), $$
since, by \eqref{cx}, we have $\sin[(m+1) z]-\sin(m z)<0$ for $z\in\big(\frac{\pi}{2m+1},\frac{3\pi}{2m+1}\big)$ and $\frac{\pi}{m+1}<\frac{3\pi}{2m+1}$ for $m\geq 3$. Thus if $\overline \upsilon_m \big(\frac{\pi}{m+1}\big)>0$ and $\overline \upsilon_m \big(\frac{\pi}{2m+1}\big)>0$ we conclude that $\overline \upsilon_m(z)>0$ for $z\in\big(\frac{\pi}{2m+1},\frac{\pi}{m+1}\big)$ and, in turn, $\upsilon_m(z)>0$ for all $z\in\big(\frac{\pi}{2m+1},\frac{\pi}{m+1}\big)$.\par 
	Recalling that $\sin\big(\frac{m\pi}{m+1}\big)=\sin\big(\frac{\pi}{m+1}\big)$ we get
	$$
	\overline \upsilon_m \bigg(\frac{\pi}{m+1}\bigg)=\sin\bigg(\dfrac{\pi}{m+1}\bigg)\bigg[\dfrac{1}{m^2}-\dfrac{1}{m^2}+\dfrac{1}{m(m+1)}\bigg]>0\qquad \forall m\geq 3.
	$$
Moreover $\sin\big(\frac{m\pi}{2m+1}\big)=\sin\big(\frac{(m+1)\pi}{2m+1}\big)$ so that
	$$
\overline \upsilon_m \bigg(\frac{\pi}{2m+1}\bigg)=\sin\bigg(\dfrac{m\pi}{2m+1}\bigg)\bigg[\dfrac{1}{m^2}-\dfrac{1}{(m+1)^2}\bigg]-\sin\bigg(\frac{\pi}{m+1}\bigg) \bigg[\dfrac{1}{m^2}-\dfrac{1}{m(m+1)}\bigg].
$$
We observe that $\sin\big(\frac{m\pi}{2m+1}\big)>\sin\big(\frac{\pi}{m+1}\big)>0$ for all $m\geq 3$, indeed $\frac{\pi}{2}>\frac{m\pi}{2m+1}>\frac{\pi}{m+1}>0$ for all $m\geq 3$; moreover we have $\frac{1}{m^2}-\frac{1}{(m+1)^2}>\frac{1}{m^2}-\frac{1}{m(m+1)}>0$, implying 	$
\overline \upsilon_m \big(\frac{\pi}{2m+1}\big)>0$.
	This concludes the proof.
\end{proof}

\par\bigskip\noindent
 \textbf{Acknowledgments.} The authors are members of the Gruppo Nazionale per l'Analisi Matematica, la Probabilit\`a e le loro Applicazioni (GNAMPA) of the Istituto Nazionale di Alta Matematica (INdAM) and are partially supported by the INDAM-GNAMPA 2019 grant: ``Analisi spettrale per operatori ellittici con condizioni di Steklov o parzialmente incernierate'' and by the PRIN project 201758MTR2: ``Direct and inverse problems for partial differential equations: theoretical aspects and applications'' (Italy).

\end{document}